\documentclass[a4paper,11pt]{amsart}
\usepackage[utf8]{inputenc}

\usepackage[
textwidth=15cm,
textheight=24cm,
hmarginratio=1:1,
vmarginratio=1:1]{geometry}

\usepackage{amssymb}

\newcommand{\D}{\mathbb D}
\newcommand{\C}{\mathbb C}
\newcommand{\T}{\mathbb T}

\newtheorem{theorem}{Theorem}
\newtheorem{lemma}[theorem]{Lemma}

\theoremstyle{remark}

\title{Rigidity of weighted composition operators on $H^p$}

\author{Mikael Lindstr\"{o}m}
\address[Lindstr\"{o}m]{Department of Mathematics, \r{A}bo Akademi University, FI-20500 \r{A}bo,
Finland}
\email{mikael.lindstrom@abo.fi}
\author{Santeri Miihkinen}
\address[Miihkinen]{Department of Mathematics, \r{A}bo Akademi University, FI-20500 \r{A}bo,
Finland}
\email{santeri.miihkinen@abo.fi}
\author{Pekka J.\ Nieminen}
\address[Nieminen]{Department of Mathematics and Statistics,
FI-20014 University of Turku, Finland}
\email{pekka.nieminen@utu.fi}


\date{13 September 2018}

\begin{document}

\begin{abstract}
We show that every non-compact weighted composition operator
$f \mapsto u\cdot (f\circ\phi)$ acting on a Hardy space $H^p$ for
$1 \leq p < \infty$ fixes an isomorphic copy of the sequence space
$\ell^p$ and therefore fails to be strictly singular. We also
characterize those weighted composition operators on $H^p$ which
fix a copy of the Hilbert space $\ell^2$. These results extend
earlier ones obtained for unweighted composition operators.
\end{abstract}

\maketitle

\section{Introduction and main results}

Let $\D$ be the open unit disc in the complex plane $\C$ and fix analytic
maps $u\colon \D \to \C$ and $\phi\colon \D \to \D$. The
weighted composition operator $uC_\phi$ is defined by
\[
   (uC_\phi) f = u \cdot (f\circ\phi)
\]
for $f\colon \D \to \C$ analytic. Boundedness and compactness
properties of such operators acting on the classical Hardy spaces
$H^p$ were characterized in terms of Carleson measures in
\cite{CHD01,CHD03} (see also \cite{CZ}). An obvious necessary
condition for the boundedness of $uC_\phi\colon H^p \to H^q$
is that $u = (uC_\phi)1 \in H^q$.

The purpose of this work is to study
the qualitative properties of non-compact weighted composition
operators on the Hardy spaces
$H^p$, extending the results obtained in \cite{LNST} for unweighted
composition operators. It turns out that the weighted composition
operators exhibit the exact same rigidity phenomena as the unweighted
ones. We also refer the reader to the recent parallel work \cite{MNST}
in the context of Volterra-type integral operators, where some of
our ideas originate from. 

Recall that if $X$ is a Banach space
and $T\colon X \to X$ is a linear operator, then $T$ is called
\emph{strictly singular} if the restriction of $T$ to any 
infinite-dimensional subspace of $X$ is not an isomorphism
(equivalently, it is not bounded below).

Our first result is a generalization of \cite[Thm~1.2]{LNST}
and shows, in particular, that the notions of compactness and strict
singularity
coincide for weighted composition operators on $H^p$. Here we employ
the usual test functions
\[
   g_a(z) = \frac{(1-|a|^2)^{1/p}}{(1-\bar{a}z)^{1/p}},
   \qquad z \in \D,
\]
where $a \in \D$. They always satisfy $\|g_a\|_{H^p} = 1$.

\begin{theorem} \label{thm:ellp}
Let $1 \leq p < \infty$ and suppose that $uC_\phi$ is bounded
and non-compact $H^p \to H^p$. Then $uC_\phi$ fixes an isomorphic copy
of $\ell^p$ in $H^p$. More precisely, there exists a sequence
$(a_n)$ in $\D$ such that $(g_{a_n})$ is equivalent to the natural
basis of $\ell^p$ and $uC_\phi$ is bounded below on the closed
linear span of $(g_{a_n})$.
\end{theorem}

We next determine under which conditions a weighted composition
operator on $H^p$ with $p\neq 2$ fixes a copy of the Hilbert space
$\ell^2$. In the unweighted case (see \cite[Thm~1.4]{LNST}) this
is the case precisely when the boundary contact set
\[
   E_\phi = \{\zeta \in\T : |\phi(\zeta)|=1\}
\]
has positive measure. It turns out that this result holds in the
weighted case as well. The first half is established in the
following theorem, where we also allow for the possibility that
the target space of the operator is a larger Hardy space than the
domain.

\begin{theorem}\label{thm:ell2}
Let $1 \leq q \leq p < \infty$ and suppose that $uC_\phi$ is bounded
$H^p \to H^q$. If $u \neq 0$ and $m(E_\phi) > 0$, then $uC_\phi$
fixes an isomorphic copy of $\ell^2$ in $H^p$.
\end{theorem}

In the converse direction we have the following result.

\begin{theorem}\label{thm:ell2conv}
Let $1 \leq p < \infty$ and suppose that $uC_\phi$ is bounded
$H^p \to H^p$ with $m(E_\phi) = 0$. If $uC_\phi$ is bounded below on
an infinite-dimensional subspace $M \subset H^p$, then $M$ contains
an isomorphic copy of $\ell^p$. In particular, if $p \neq 2$,
then $uC_\phi$ does not fix any isomorphic copies of $\ell^2$ in
$H^p$.
\end{theorem}

The last statement of the theorem is due to the fact that
$\ell^p$ and $\ell^2$ are totally incomparable spaces for $p \neq 2$.

\section{Proofs}

Towards the proof of Theorem~\ref{thm:ellp} we first state the
following lemma.

\begin{lemma}\label{le:Hump}
Let $u \in H^p$ and $\phi\colon \D \to \D$ be analytic. For
$\epsilon > 0$, define
$$F_\epsilon = \{\zeta\in\T: |\phi(\zeta)-1| < \epsilon\}.$$ Then
\begin{gather*}
   \lim_{a\to 1} \int_{\T\setminus F_\epsilon}
     |(uC_\phi)g_a|^p\,dm = 0 \quad\text{for each $\epsilon > 0$,} \\
   \lim_{\epsilon\to 0} \int_{F_\epsilon}
     |(uC_\phi)g_a|^p\,dm = 0 \quad\text{for each $a \in \D$.} \\
\end{gather*}
\end{lemma}

\begin{proof}
Let $\epsilon > 0$ be fixed and consider
$\zeta \in \T\setminus F_\epsilon$ for which the radial limit
$\phi(\zeta)$ exists. Then, if $|a-1| < \epsilon/2$, we have
\[
   |1-\bar{a}\phi(\zeta)| \geq |1-\phi(\zeta)| - |1-a| > \epsilon/2,
\]
and so
\[
   \int_{\T\setminus F_\epsilon} |(uC_\phi)g_a|^p\,dm
   \leq (1-|a|^2) \int_{\T\setminus F_\epsilon}
        \frac{|u|^p}{|1-\bar{a}\phi|^2}\,dm
   \leq \frac{4(1-|a|^2)}{\epsilon^2} \|u\|_{H^p}^p.
\]
Since this tends to $0$ as $a \to 1$, we obtain the first part of the
lemma.

The second part follows from the absolute continuity of the measure
$F \mapsto \int_F |(uC_\phi)g_a|^p\,dm$ and the fact that
$m(F_\epsilon) \to m(\{\zeta\in\T: \phi(\zeta)=1\}) = 0$ as
$\epsilon \to 0$. Note that $g_a \in H^\infty$ and hence
$(uC_\phi)g_a \in L^p(\T,m)$.
\end{proof}

\begin{proof}[Proof of Theorem~\ref{thm:ellp}]
Since $uC_\phi$ is non-compact, we may find a sequence $(a_n)$ in
$\D$ such that $|a_n| \to 1$ and
$\|(uC_\phi)g_{a_n}\|_{H^p} \geq c > 0$ for all $n$. This is a
consequence of the compactness characterization of $uC_\phi$ in
terms of vanishing Carleson measures; see \cite[Theorem~3.5]{CHD01}.
By passing to a convergent subsequence of $(a_n)$ and utilizing a
suitable rotation, we may assume that $a_n \to 1$.

We now proceed exactly as in the unweighted case (see the proof of
Theorem~1.2 in \cite{LNST} for the details of the following
argument). First, by invoking
Lemma~\ref{le:Hump} repeatedly, we may extract a subsequence of
$(a_n)$, still denoted by $(a_n)$, such that the image sequence
$((uC_\phi)g_{a_n})$ in $H^p$ is equivalent to the standard basis
of $\ell^p$, that is,
\[
   \biggl\|\sum_{n=1}^\infty \alpha_n (uC_\phi)g_{a_n} \biggr\|_{H^p}
   \sim \|(\alpha_n)\|_p
   \quad\text{for $(\alpha_n) \in \ell^p$.}
\]
Then a second application of Lemma~\ref{le:Hump} to the functions
$g_{a_n}$ (taking $u = 1$ and $\phi(z) = z$) produces a further
subsequence of $(a_n)$, which we continue to denote by $(a_n)$,
such that also
 \[
   \biggl\|\sum_{n=1}^\infty \alpha_n g_{a_n} \biggr\|_{H^p}
   \sim \|(\alpha_n)\|_p
   \quad\text{for $(\alpha_n) \in \ell^p$.}
\]
By combining the preceding two norm estimates we see that
$uC_\phi$ restricts to a linear isomorphism on the closed linear
span of $(g_{a_n})$.
\end{proof}

\begin{proof}[Proof of Theorem~\ref{thm:ell2}]
Since $m(E_\phi) > 0$, \cite[Prop.~3.2]{LNST} shows that there exists
a sequence of integers
$(n_k)$ satisfying $\inf_k (n_{k+1}/n_k) > 1$
and a constant $c_1 > 0$ such that
$\bigl\|\sum_k\alpha_k\phi^{n_k}\bigr\|_{H^1}
\geq c_1\|(\alpha_k)\|_2$ for all $(\alpha_k) \in \ell^2$.
Our goal is to prove a weighted version of this estimate,
that is, for some constant $c > 0$,
\begin{equation}\label{eq:weightedH1}
   \biggl\| u \sum_k \alpha_k \phi^{n_k} \biggr\|_{H^1}
   \geq c \|(\alpha_k)\|_2.
\end{equation}
Since Paley's theorem (see e.g.\ \cite[p.~104]{Duren}) implies that
the closed linear
span $M = [z^{n_k}: k \geq 1]$ in $H^p$ is isomorphic to
$\ell^2$, inequality \eqref{eq:weightedH1} implies that
$uC_\phi$ is an isomorphism from $M$ into $H^1$. This yields the
theorem because $\|f\|_{H^q} \geq \|f\|_{H^1}$ for all $f \in H^q$.

To establish \eqref{eq:weightedH1}, we first note that
since $u \neq 0$, we have $|u| > 0$ a.e.\ on $\T$. Thus, for
a given $\epsilon > 0$ there
exist a set $F \subset \T$ with
$m(\T\setminus F) < \epsilon$ such that
$|u| > c_2$ on $F$ for some constant $c_2 = c_2(\epsilon) > 0$.
Then, using Hölder's inequality and the boundedness of $C_\phi$
on $H^2$, we get
\[
   \int_{\T\setminus F} \biggl| \sum_k \alpha_k\phi^{n_k} \biggr|\,dm
   \leq \sqrt{m(\T\setminus F)}
        \biggl\|\sum_k \alpha_k\phi^{n_k} \biggr\|_2
   \leq \sqrt{\epsilon} \cdot c_3\|(\alpha_k)\|_2
\]
for some constant $c_3 > 0$. On combining these estimates we obtain
\[
   \biggl\| u \sum_k \alpha_k \phi^{n_k} \biggr\|_{H^1}
   \geq c_2 \int_F \biggl|\sum_k \alpha_k \phi^{n_k} \biggr|\,dm
   \geq c_2 (c_1 - \sqrt{\epsilon}\cdot c_3) \|(\alpha_k)\|_2.
\]
In particular, choosing
$\epsilon = (c_1/2c_3)^2$ here proves \eqref{eq:weightedH1} with
$c = \frac{1}{2} c_2c_1$. This completes the proof of the theorem.
\end{proof}

\begin{proof}[Proof of Theorem~\ref{thm:ell2conv}]
Since $M$ is infinite-dimensional, there exists a sequence $(f_n)$
in $M$ such that $\|f_n\|_{H^p} = 1$ and $f_n \to 0$ uniformly on
compact subsets of $\D$; for instance, we can choose
$f(0) = f'(0) = \cdots = f^{(n)}(0) = 0$ for all $n$.

For each $k \geq 1$, define
$E_k = \{\zeta \in \T: |\phi(\zeta)| > 1-\tfrac{1}{k}\}$. We have
$m(E_k) \to m(E_\phi) = 0$ as $k \to \infty$ and therefore
\begin{equation}\label{eq:limk}
   \lim_{k\to\infty} \int_{E_k} |(uC_\phi)f_n|^p\,dm = 0.
\end{equation}
On the other hand, since $f_n\circ\phi$ converges to zero uniformly
on $\T\setminus E_k$ as $n\to\infty$ and $u \in H^p$,
\begin{equation}\label{eq:limn}
   \lim_{n\to\infty} \int_{\T\setminus E_k} |(uC_\phi)f_n|^p\,dm = 0.
\end{equation}
Since $uC_\phi$ is bounded below on $M$, we also have
$\|(uC_\phi)f_n\|_{H^p} \geq c$ for all $n$ and some constant
$c > 0$. Using a gliding hump argument based on a repeated
application of \eqref{eq:limk} and \eqref{eq:limn} (akin to
the proof of \cite[Prop.~3.3]{LNST} in the unweighted case),
we may extract a subsequence $(f_{n_j})$ such that the
sequence $\bigl((uC_\phi)f_{n_j}\bigr)$ is equivalent to the
standard basis of $\ell^p$. Since $uC_\phi$ is bounded below
on the closed linear span $[f_{n_j}: j \geq 1] \subset M$, we
conclude that $uC_\phi$ fixes a copy of $\ell^p$ in $M$.
\end{proof}



\begin{thebibliography}{99}

\bibitem{CHD01}
M.D.~Contreras and A.G.~Hernández-Díaz,
\emph{Weighted composition operators on Hardy spaces},
J.\ Math.\ Anal.\ Appl.\ 263 (2001), no.~1, 224--233.

\bibitem{CHD03}
M.D.~Contreras and A.G.~Hernández-Díaz,
\emph{Weighted composition operators between different Hardy spaces},
Integral Equations Operator Theory 46 (2003), 165--188.

\bibitem{CZ}
Z.~\u{C}u\u{c}ković and R.~Zhao,
\emph{Weighted composition operators between different weighted
Bergman spaces and different Hardy spaces},
Illinois J.\ Math.\ 51 (2007), no.~2, 479--498.

\bibitem{Duren}
P.L.~Duren, \emph{Theory of $H^p$ Spaces}, Academic Press, 1970;
reprinted by Dover, 2000.

\bibitem{LNST}
J.~Laitila, P.J.~Nieminen, E.~Saksman and H.-O.~Tylli,
\emph{Rigidity of composition operators on the Hardy space $H^p$},
Adv.\ Math.\ 319 (2017), 610--629.

\bibitem{MNST}
S.~Miihkinen, P.J.~Nieminen, E.~Saksman and H.-O.~Tylli,
\emph{Structural rigidity of generalised Volterra operators on $H^p$},
Bull.\ Sci.\ Math.\ 148 (2018), 1--13.


\end{thebibliography}
\end{document}